\documentclass[12pt]{amsart}
\font\bbbld=msbm10 scaled\magstephalf

\newcommand{\bfR}{\hbox{\bbbld R}}

\newcommand{\e}{\varepsilon}
\newcommand{\goto}{\rightarrow}

\newcommand{\be}{\begin{equation}}
\newcommand{\ee}{\end{equation}}
\newcommand{\bea}{\begin{eqnarray}}
\newcommand{\eea}{\end{eqnarray}}

\newtheorem{theorem}{Theorem}[section]
\newtheorem{lemma}[theorem]{Lemma}
\newtheorem{proposition}[theorem]{Proposition}
\newtheorem{corollary}[theorem]{Corollary}

\theoremstyle{definition}

\theoremstyle{remark}
\newtheorem{remark}[theorem]{Remark}

\numberwithin{equation}{section}

\begin{document}
\setlength{\baselineskip}{1.2\baselineskip}

\title
[A priori estimates for semistable solutions]
{A priori estimates for semistable solutions of semilinear elliptic equations }

\author{Xavier Cabr\'{e}}
\address{ICREA and Departament de Matem\`{a}tica Aplicada I,  Universitat Polit\`{e}cnica de Catalunya, 08028 Barcelona , Spain}
\email{xavier.cabre@upc.edu}
\author{Manel Sanch\'{o}n}
\address{Departament de Matem\`{a}tiques, Universitat Aut\`onoma de Barcelona, 08193 Bellaterra, Spain}
\email{msanchon@mat.uab.cat}
\author{Joel Spruck}
\address{Department of Mathematics, Johns Hopkins University,
 Baltimore, MD 21218}
\email{js@math.jhu.edu}
\thanks{Research of the first and second authors supported by MTM2011-27739-C04-01
(Spain) and 2009SGR345 (Catalunya). The second author is also supported by ERC grant 320501 
(ANGEOM project). Research of the third author supported in part by the NSF and Simons 
Foundation.}

%
%
%

\maketitle
\vspace {-.2in}
\begin{center}
{\em In memory of Rou-Huai Wang}
\end{center}

\section{Introduction}\label{sec1}

\setcounter{equation}{0}
In this note we consider semistable solutions of the boundary value problem
\begin{equation} \label{eq1.10}
\left\{
\begin{array}{rclll}
 Lu+f(u)&=&0&\textrm{in}&\Omega,
\\ 
u&=&0&\textrm{on}&\partial\Omega,
\end{array}
\right.
\end{equation}
where $\Omega \subset \bfR^n$ is a smooth bounded domain with $n\geq 2$, $f\in C^2$, and 
$Lu:=\partial_i(a^{ij}(x)u_j)$ is uniformly elliptic. More 
precisely, we assume that $(a^{ij}(x))$ is a symmetric $n\times n$ matrix with bounded
measurable coefficients, \textit{i.e.}, $a^{ij}=a^{ji}\in L^\infty(\Omega)$, for which
there exist positive constants $c_0$ and $C_0$ satisfying
\begin{equation}\label{uniform-ellipticity}
c_0|\xi|^2\leq a^{ij}(x)\xi^i\xi^j\leq C_0|\xi|^2\quad\textrm{for all }
\xi\in\mathbb{R}^n,\ x\in \Omega.
\end{equation}

By semistability of the solution $u$, we mean that the lowest Dirichlet eigenvalue of the 
linearized operator at $u$ is nonnegative. That is, we have the 
{\em semistability inequality}
\begin{equation} \label{eq1.20}
\int_{\Omega} f'(u)\eta^2~dx 
\leq 
\int_{\Omega} a^{ij}(x)\eta_i \eta_j~dx \quad\textrm{for all } \eta \in H^1_0(\Omega).
\end{equation}

There is a large literature on a priori estimates, beginning with the seminal 
paper of Crandall and Rabinowitz \cite{CR}. In \cite{CR} and subsequent works, 
a basic and  standard assumption is that $u$ is positive in $\Omega$ and 
$f\in C^2$ is positive, nondecreasing, and superlinear at infinity: 
\begin{equation} \label{eq1.30} 
f(0) > 0, \quad f'\geq 0 \quad\textrm{and}\quad\lim_{t\goto +\infty}\frac{f(t)}t=\infty.
\end{equation}

Note that, under these assumptions and with $f(u)$ replaced by $\lambda f(u)$ with $\lambda\geq 0$, 
semistable solutions do exist for an interval of parameters $\lambda\in(0,\lambda^*)$; see \cite{CR}.

In recent years there have been strong efforts to obtain a priori bounds under minimal 
assumptions on $f$ (essentially \eqref{eq1.30}), mainly after  
 Brezis and V\'azquez \cite{BV} raised several open questions.
The following are the main results in this direction.
The important paper of Nedev \cite{Nedev} obtains the $L^\infty$ bound for
$n=2$ and $3$ if $f$ satisfies \eqref{eq1.30} and in addition $f$ is convex. 
Nedev states his result for $L=\Delta$ but it is equally valid for general $L$. 
When $2\leq n \leq 4$ and $L=\Delta$, Cabr\'{e} \cite{Cabre} established that the 
$L^{\infty}$ bound holds for  {\em  arbitrary $f$} if in addition $\Omega$ convex.
Villegas \cite{V} replaced the condition that $\Omega$ is convex 
in Cabr\'{e}'s result assuming instead that $f$ is convex. For the radial case, 
Cabr\'e and Capella~\cite{CC} proved the $L^\infty$ bound when $n\leq 9$. On the other hand, 
it is well known that there exist unbounded semistable solutions when $n\geq 10$ 
(for instance, for the exponential nonlinearity $e^u$). 

For convex nonlinearities $f$ and under extra assumptions involving the two numbers
\begin{equation}\label{tau}
\tau_- := \liminf_{t \goto \infty} \frac{f(t)f''(t)}{f'(t)^2}
\leq
\tau_+ := \limsup_{t \goto \infty} \frac{f(t)f''(t)}{f'(t)^2}
\end{equation}
much more is known (see more detailed comments after Corollary \ref{cor}). For instance, 
Sanch\'on~\cite{Sanchon} proved that $u\in L^\infty(\Omega)$ whenever $\tau_-=\tau_+\geq 0$ and 
$n\leq 9$. This hypothesis is satisfied by $f(u)=e^u$, as well as by $f(u)=(1+u)^m$, $m>1$.

It is still an \textbf{open problem} to establish an $L^\infty$ estimate in general domains $\Omega$ when $n\leq 9$ 
under \eqref{eq1.30} as the only assumption on $f$. 

Our purpose here is to prove the following results:
\begin{theorem} \label{th1}  
Let $ f \in C^ 2$  be convex and satisfy \eqref{eq1.30}.  
Assume  in  addition that for every $\varepsilon > 0$, there exist 
$T = T (\varepsilon)$ and $C=C(\varepsilon) $ such that
\begin{equation} \label{eq1.40} 
f '(t) \leq C f(t)^{1+\varepsilon}\quad   \text{for all } t > T. 
\end{equation}
Then if $u$ is a positive semistable solution of \eqref{eq1.10}, we have 
$ f'(u)\in L^p(\Omega)$ for all $p<3$ and $n\geq 2$, while $f(u) \in L^p (\Omega)$  
for all  $p < \frac{n}{n-4}$ and $n\geq 6$. 

As a consequence, we deduce respectively:
\\
$(a)$ If $n \leq 5$, then $u\in L^{\infty}(\Omega)$.
\\
$(b)$ If $n \geq  6$,  then $ u
\in W^{1,p}_0(\Omega)$ for all $p < \frac{n}{n-5}$ and $u \in
L^p(\Omega)$ \text{for all} $p < \frac{n}{n-6}$. 
In particular, if $n \leq 9$ then $u\in H^1_0(\Omega)$.\end{theorem}

Theorem \ref{th1} establishes the $L^\infty$ bound up to dimension $5$ when 
\eqref{eq1.40} holds. If we assume more about $f$ we can obtain an 
$L^{\infty}$ bound up to dimension $n=6$.

 \begin{theorem} \label{th2}    
Let $ f \in C^ 2$  be convex and satisfy \eqref{eq1.30}. Assume in addition that 
there exist $\varepsilon \in (0,1)$ and $T = T (\varepsilon) $ such that
\begin{equation}\label{H:tm2} 
f '(t) \leq C f(t)^{1-\varepsilon} \quad  \text{for all } t > T.
\end{equation}
Then  if $u$ is a positive semistable solution of \eqref{eq1.10}, we have $f'(u) 
\in L^\frac{3- \varepsilon}{1-\varepsilon}(\Omega)$ for all $n\geq 2$, while 
$ f (u) \in L^p (\Omega)$ for all $p < \frac{(1-\varepsilon)n}{(1-\varepsilon)n-4+2\varepsilon}$
and $n\geq  6+\frac{4\varepsilon}{1-\varepsilon}$.

As a consequence, we deduce respectively:
\\
$(a)$ If $n<6+\frac{4\varepsilon}{1-\varepsilon}$, then $u\in L^{\infty}(\Omega)$.
\\
$(b)$ If $n \geq  6+\frac{4\varepsilon}{1-\varepsilon}$,  
then $ u\in W^{1,p}_0(\Omega)$ for all $p < \frac{(1-\varepsilon)n}{(1-\varepsilon)n-5+3\varepsilon}$
and $u\in L^p(\Omega)$ for all $p < \frac{(1-\varepsilon)n}{(1-\varepsilon)n-6+4\varepsilon}$. 
In particular,  if $ n < 10+\frac{4\varepsilon}{1-\varepsilon}$ then $u\in H^1_0(\Omega)$.
\end{theorem}

The main novelty of our results are twofold. On the one hand, we do not assume any lower 
bound on $f'$ to obtain our estimates, nor any bound on $f''$ 
as in \cite{CR} or \cite{Sanchon} (as commented below). On the other hand, we obtain $L^p$
estimates for $f'(u)$. To our knowledge such estimates do not exist in the literature. 
In fact, using the $L^p$ estimate for $f(u)$ established in Theorem~\ref{th2} and standard 
regularity results for uniformly elliptic equations, it follows that $u$ is bounded in 
$L^\infty(\Omega)$ whenever $n<6+\frac{2\varepsilon}{1-\varepsilon}$. 
Note that the range of dimensions obtained in Theorem~\ref{th2}~(a), 
$n<6+\frac{4\varepsilon}{1-\varepsilon}$, is bigger than this one. 
This will follow from the $L^p$ estimate on $f'(u)$.
Of course, in both results (and also in the rest of the paper), $u\in L^p$ 
or $u\in W^{1,p}$ mean that $u$ is bounded in $L^p$ or in $W^{1,p}$ by a 
constant independent of $u$.

Our assumptions \eqref{eq1.40} and \eqref{H:tm2} in Theorems \ref{th1} and \ref{th2} are related to 
the hypothesis $\tau_+\leq 1$ (recall \eqref{tau}). 
Indeed, by the definition of $\tau_+$,  for every $\delta>0$
there exists $T=T(\delta)$ such that $f(t)f''(t) \leq (\tau_+ + \delta)f'(t)^2$ for all $t > T$, or equivalently,
$\frac{d}{dt} \frac{f'(t)}{f(t)^{\tau_++\delta}}\leq 0$ for all $t>T$. Thus,
\[
\frac{f'(t)}{f(t)^{\tau_+ + \delta}}
\leq 
\frac{f'(T)}{f(T)^{\tau_+ + \delta}}=C
\quad \textrm{for all } t > T.
\]
From this, it is clear that if $\tau_+\leq 1$, then assumption \eqref{eq1.40} holds 
choosing $\delta=\varepsilon$. Note that $0\leq \tau_-\leq 1$ always holds  since $f$ is a 
continuous function defined in $[0,+\infty)$. 
If instead $\tau_+<1$, then \eqref{H:tm2} is satisfied 
with $\varepsilon=1-\tau_+-\delta$, where $\delta>0$ is arbitrarily small. Therefore as 
an immediate consequence of part (i) of Theorems~\ref{th1} and  \ref{th2} we obtain the 
following.

\begin{corollary}\label{cor}
Let $ f \in C^ 2$  be convex and satisfy \eqref{eq1.30}. Let $u$ be a positive semistable 
solution of \eqref{eq1.10}. The following assertions hold:
\\
$(a)$ If $\tau_+=1$ and $n<6$ then $u\in L^\infty(\Omega)$.
\\
$(b)$ If $\tau_+<1$ and $n<2+\frac{4}{\tau_+}$ then $u\in L^\infty(\Omega)$.
\end{corollary}

If $\tau_+<1$, then for every $\varepsilon\in(0,1-\tau_+)$, there exists a positive constant $C$ such that 
$f(t)\leq C(1+t)^\frac{1}{1-\tau_+-\varepsilon}$ for all $t\geq 0$ (this can be easily 
seen integrating twice in the definition of $\tau_+$). Thus under this hypothesis, $f$ has at most polynomial growth.

All the results in the literature considering $\tau_-$ and $\tau_+$ (defined in \eqref{tau})
assume $\tau_->0$. Instead in Corollary \ref{cor}, no assumption is made on $\tau_-$.

Crandall and Rabinowitz \cite{CR} proved an a priori $L^{\infty}$ 
bound for semistable solutions when $0<\tau_- \leq \tau_+ <2+\tau_-+2\sqrt{\tau_-}$ 
and $n<4+2\tau_-+4\sqrt{\tau_-}$. Note that for nonlinearities $f$ such that $\tau_-=1$ and $\tau_+<5$ one obtains the 
$L^\infty$ bound if $n\leq 9$ (a dimension which is optimal). This is the case for many exponential type 
nonlinearities, as for instance $f(u)=e^{u^\alpha}$ for any $\alpha\in\mathbb{R}^+$. 
The results in \cite{CR} were improved in \cite{Sanchon} establishing that 
$u\in L^\infty(\Omega)$ whenever $\tau_->0$ and $n<6+4\sqrt{\tau_-}$ (remember that $\tau_-\leq 1$). 
Moreover, if $0<\tau_-\leq\tau_+<1$, then using an iteration argument in \cite{CR}, one has that 
$u\in L^\infty(\Omega)$ whenever $n<2+\frac{4}{\tau_+}\left(1+\sqrt{\tau_-}\right)$. 
Note that Corollary \ref{cor} coincides with these results in the case where $\tau_-=0$.

Let us make some further comments on conditions \eqref{eq1.40}
and \eqref{H:tm2} in Theorems~\ref{th1} and \ref{th2}, respectively.

\begin{remark}\label{rk1}
$(i)$ Condition \eqref{eq1.40} is equivalent to 
$$
\limsup_{t\rightarrow+\infty}\frac{{\rm log} f'(t)}{{\rm log} f(t)}\leq 1,
$$
since \eqref{eq1.40} holds if and only if
$$
\frac{{\rm log} f'(t)}{{\rm log} f(t)}\leq (1+\varepsilon)+\frac{C}{{\rm log} f(t)}
\quad\textrm{for all }t>T;
$$
note that $f(t)\rightarrow+\infty$ as $t\rightarrow+\infty$ by \eqref{eq1.30}.
Many nonlinearities $f$ satisfy this condition (like exponential or power type 
nonlinearities).

$(ii)$ Setting $s=f(t)$ and $t=\gamma(s)$, \eqref{eq1.40} is equivalent to the condition 
$
\gamma'(s) \geq \theta s^{-1-\varepsilon}
$ 
for some $\theta>0$ and for all $s$ sufficiently large. 
This clearly shows that \eqref{eq1.40} does not follow from the convexity of $f$
alone (which is equivalent to $\gamma'$ being nonincreasing). 

Instead, condition \eqref{H:tm2} is equivalent 
to $\gamma'(s)\geq \theta s^{-1+\varepsilon}$ for some $\theta>0$ and for all $s$ sufficiently 
large. In particular, $\gamma(s)\geq \theta s^\varepsilon$ for $s$ large enough, or equivalently,
$f(t)\leq C (1+t)^\frac{1}{\varepsilon}$ for some constant $C>0$ and for all $t$.

On the other hand, $f(t)f''(t)/f'(t)^2=-s\gamma''(s)/\gamma'(s)$, a second derivative condition on 
$\gamma$, in contrast with the first derivative conditions of \eqref{eq1.40} and \eqref{H:tm2}. Therefore, for most 
nonlinearities satisfying \eqref{eq1.40} (or \eqref{H:tm2}), the limit $ff''/(f')^2$ at infinity does not exist 
(\textit{i.e.}, $\tau_-<\tau_+$) and in addition, it may happen that $\tau_-=0$.

$(iii)$ Note that by convexity, $\varepsilon f'(t) \leq f(t+\varepsilon)-f(t)
\leq f(t+\varepsilon)$ for all $t$. Therefore, \eqref{eq1.40} holds if 
$f(t+\varepsilon)\leq Cf(t)^{1+\varepsilon}$ for all $t$ sufficiently large.
\end{remark}

\section{Preliminary estimates}
\label{sec2}
\setcounter{equation}{0}


We start by recalling the following standard regularity result for uniformly elliptic 
equations.

\begin{proposition}\label{lem2:1}
Let $a^{ij}=a^{ji}$, $1\leq i,j\leq n$, be measurable functions on
a bounded domain $\Omega$. Assume that there exist positive constants $c_0$
and $C_0$ such that \eqref{uniform-ellipticity} holds. Let $u\in H^1_0(\Omega)$
be a weak solution of 
$$
\left\{
\begin{array}{rccl}
Lu+c(x)u&=&g(x)&\textrm{in }\Omega,\\
u&=&0&\textrm{on }\partial \Omega,
\end{array}
\right.
$$
with $c$, $g\in L^p(\Omega)$ for some $p\geq 1$. 

Then, there exists a positive constant $C$ independent of $u$ such that 
the following assertions hold:

$(i)$ If $p>n/2$ then $\|u\|_{L^\infty(\Omega)}\leq C(\|u\|_{L^1(\Omega)}+\|g\|_{L^p(\Omega)})$.

$(ii)$ Assume $c\equiv 0$. If $1\leq p<n/2$ then
$
\|u\|_{L^{r}(\Omega)}\leq C\|g\|_{L^p(\Omega)}
$
for every $1\leq r<np/(n-2p)$.
Moreover, $\|u\|_{W^{1,r}_0(\Omega)}\leq C$ for every $1\leq r <np/(n-p)$.
\end{proposition}

Part (i) of Proposition~\ref{lem2:1} is established in Theorem~3 of \cite{Serrin} 
with the $L^2$-norm of $u$ instead of the $L^1$-norm. However, an immediate interpolation 
argument shows that the result also holds with $\|u\|_{L^1(\Omega)}$. Note also 
that in the right hand side of this estimate, $\|u\|_{L^\infty(\Omega)}\leq
C(\|u\|_{L^1(\Omega)}+\|g\|_{L^p(\Omega)})$,  some dependence of $u$ must appear (think 
on the equation with $g\equiv 0$ satisfied by the eigenfunctions of the Laplacian).
For part~(ii) we refer to Theorems~4.1 and 4.3 of \cite{Trudinger}.

As an easy consequence of Proposition~\ref{lem2:1}~(i) we obtain the following:
\begin{corollary} \label{cor2.4} 
Let  $u\in H^1_0(\Omega)$ be a nonnegative weak solution of \eqref{eq1.10} 
with $f$ nondecreasing and convex. Assume $p>n/2$.
If  there exists a positive constant $C$ independent of $u$ such that 
$\|u\|_{L^1(\Omega)}\leq C$ and $\|f'(u)\|_{L^p (\Omega)}\leq C$,  
then $\|u\|_{L^\infty(\Omega)}\leq C$ for some positive constant $C$ independent of $u$.
\end{corollary}

\begin{proof} 
Rewrite equation \eqref{eq1.10} as $Lu +c(x)u=-f(0)$ where $c(x)=(f(u)-f(0))/u$. 
Then by convexity,  $0\leq c(x)\leq f'(u)$ and the result follows by Proposition~\ref{lem2:1}~(i).
\end{proof}

The following estimates involving
$$
\tilde{f}(u):=f(u)-f(0)
$$
are due to Nedev \cite{Nedev} when $L=\Delta$. 
We give here a new proof of the estimates consistent with our own approach. 
Note that assumptions \eqref{eq1.40} and \eqref{H:tm2}
in Theorems~\ref{th1} and \ref{th2}, respectively, also hold replacing $f$ by $\tilde{f}$ on their right hand side, since 
$f(t)\leq 2(f(t)-f(0))=2\tilde{f}(t)$ for $t$ large enough. We will use this fact in the proof 
of both results.

\begin{lemma} \label{lem1} 
Let $ f \in C^ 2$  be convex and satisfy \eqref{eq1.30}. 
If $u$ is a positive semistable solution of \eqref{eq1.10}, 
then there exists a positive constant $C$ independent of $u$ such that
\begin{equation} \label{eq2.27} 
\int_{\Omega} \tilde{f}(u)f'(u)~dx \leq C
\quad\textrm{and}\quad 
\int_{\Omega} \tilde{f}(u) f''(u) a^{ij}(x)u_i u_j ~dx \leq C.
\end{equation}
\end{lemma}
\begin{proof}
Let $u$ be a semistable solution of \eqref{eq1.10} and $\tilde{f}(u) = f (u) - f(0)$. Note 
that $\tilde{f}(u)$ satisfies 
$$
L (\tilde{f}(u))+ f'(u) \tilde{f} (u)   = -f (0)f' (u) +f '' (u)a^{ij}(x)u_i u_j .
$$
Multiplying the previous identity by $\tilde{f}(u)$ and using the semistability condition 
\eqref{eq1.20},  we obtain
$$
\begin{aligned}
 0 &  \leq \int_{\Omega} (a^{ij}(\tilde{f}(u))_i (\tilde{f}(u))_j -f'(u)\tilde{f}(u)^2)~ dx \\
 &= f (0)\int_{\Omega} f'(u)\tilde{f}(u)~dx-\int_{\Omega}
\tilde{f}(u) f''(u) a^{ij}u_i u_j ~dx,
\end{aligned}
$$
or equivalently,
\begin{equation} \label{eq2.25}
\int_{\Omega}
\tilde{f}(u) f''(u) a^{ij}(x)u_i u_j ~dx \leq f (0)\int_{\Omega} \tilde{f}(u)f'(u)~dx.
\end{equation}
As a consequence, the second estimate in \eqref{eq2.27} follows by the first one.

Multiplying the equation \eqref{eq1.10} by the test functions
$\zeta=f'(u)-f'(0)$ and 
$$
\zeta
=\left\{ 
\begin{array}{cll} 
0 &\textrm{if}&u\leq M
\\
f'(u)-f'(M) &\textrm{if}&u> M,
\end{array} 
\right.
$$
we find
\begin{equation} \label{eq2.30} 
\int_{\Omega}f''(u) a^{ij}(x)u_i u_j~dx= \int_{\Omega} f(u)(f'(u)-f'(0))~dx 
\end{equation}
and
\begin{equation} \label{eq2.40} 
\int_{ \{u>M\}}f''(u) a^{ij}(x)u_i u_j~dx= \int_{\{u>M\}} f(u)(f'(u)-f'(M))~dx,
\end{equation}
respectively.

Combining \eqref{eq2.25} and \eqref{eq2.30}, we obtain
\begin{equation} \label{eq2.50} 
\int_{\Omega}(f(u)-2f(0))f''(u)a^{ij}(x)u_i u_j~dx
\leq 
f(0)f'(0) \int_{\Omega} f(u)~dx
-
f(0)^2\int_{\Omega}f'(u)~dx. 
\end{equation}
Choose $M$ (depending on $f$) such that $f(t)>2f(0)+2$ for all $t\geq M$. 
On the one hand, using \eqref{eq2.40}, the convexity of $f$, 
and that  $(a^{ij})$ is a positive definite matrix, we obtain
$$
\begin{array}{lll}
\displaystyle 2\int_{\{u>M\}}f(u)(f'(u)-f'(M))~dx
&=&
\displaystyle 2\int_{\{u>M\}}f''(u)a^{ij}(x)u_iu_j~dx
\\
&\leq&
\displaystyle \int_{\{u>M\}}(f(u)-2f(0))f''(u)a^{ij}(x)u_iu_j~dx.
\end{array}
$$
On the other hand, for some constant $C$ depending only on $f$ (and $M$),  
there holds
$$
\begin{array}{lll}
\displaystyle -\int_{\{u\leq M\}}(f(u)-2f(0))f''(u)a^{ij}(x)u_iu_j~dx
&\leq& 
\displaystyle C\int_{\{u\leq M\}}a^{ij}(x)u_iu_j~dx
\\
&\leq&
\displaystyle CM\int_\Omega f(u)~dx,
\end{array}
$$
where the last inequality follows from multiplying equation \eqref{eq1.10} by 
$\min\{u,M\}$. Combining the previous bounds with
\eqref{eq2.50}, it follows that
\begin{equation} \label{eq2.60}
2\int_{\{u>M\}}f(u)(f'(u)-f'(M))~dx
\leq
f(0)f'(0)\int_\Omega f(u)~dx + CM\int_\Omega f(u)~dx.
\end{equation}

Finally, choose $\overline{M}>M$ (depending only on $f$) such that $f'(M)<\frac{f'(t)}2$ if 
$t>\overline{M}$. Then \eqref{eq2.60} implies
$$
 \int_{\{u>\overline{M}\}} f(u)f'(u)~dx\leq C\int_{\Omega} f(u)~dx,
$$
and using that $f'(t)\rightarrow+\infty$ at infinity (see Remark \ref{remL1} below), 
we conclude $$\int_\Omega f(u)f'(u)~dx\leq C~,$$ where $C$ is independent of $u$.
\end{proof}

\begin{remark}\label{remL1} 
Note that $\tilde{f}(t)/t \leq f'(t)$ for all $t\geq 0$ since $f$ is convex. In particular, by
condition \eqref{eq1.30}, we obtain $\lim_{t\goto \infty}f'(t)=\infty$. Therefore, as a consequence 
of estimate \eqref{eq2.27} we obtain 
\begin{equation}\label{L1}
\int_{\Omega}f(u)dx\leq C, 
\end{equation}
where $C$ is a constant independent of $u$. As in \cite{Nedev}, from this and Proposition~\ref{lem2:1}~(ii), 
one deduces
\begin{equation}\label{L1:bis}
u \textrm{ in }L^q(\Omega) \textrm{ for all  }q<n/(n-2).
\end{equation}
Our results improve this estimate under the additional assumptions on $f$ 
of Theorems~\ref{th1} and \ref{th2}.

The following is a sufficient condition on $f$ to guarantee $u\in H^1_0(\Omega)$. 
Note that by convexity of $f$, $tf'(t)-\tilde{f}(t)\geq 0$ for all $t\geq 0$.
If we further assume that for some $\e>0$, $tf'(t)-\tilde{f}(t) \geq \varepsilon t$ for $t>T(\varepsilon)$, 
then $u$ is bounded in $H^1_0(\Omega)$ by a constant independent of $u$.  Indeed, 
noting that
$$
\begin{array}{lll}
\displaystyle \varepsilon\int_\Omega a^{ij}(x)u_iu_j~dx
&=&
\displaystyle \varepsilon\int_ \Omega f(u)u~dx
\leq C+ \int_ \Omega f(u)(u f'(u)- \tilde{f}(u))~dx
\\
&=&
\displaystyle C+\int_\Omega a^{ij}(x)u_i (u f'(u)- \tilde{f}(u))_j ~dx\\
&=&
\displaystyle C+\int_\Omega u f''(u) a^{ij}(x)u_i u_j ~dx
\leq
C,
\end{array}
$$
where in the last inequality we used the superlinearity of $f$ and the second 
estimate in \eqref{eq2.27}.
\end{remark}

\section{Proof of Theorems \ref{th1} and \ref{th2}}
\label{sec3}
\setcounter{equation}{0}

\begin{proof}[Proof of Theorem~{\rm \ref{th1}}]
Assume \eqref{eq1.40}. In fact, as we said before Lemma~\ref{lem1} we
may assume that \eqref{eq1.40} holds replacing $f$ by $\tilde{f}$:
for every $\varepsilon > 0$, there exist 
$T = T (\varepsilon)$ and $C=C(\varepsilon) $ such that
\begin{equation} \label{eq1.40:bis} 
f '(t) \leq C \tilde{f}(t)^{1+\varepsilon}\quad   \text{for all } t > T. 
\end{equation}
In the following, the constants $C$ may depend on $\varepsilon$ and $T$ but 
are independent of $u$.

We start by proving that $f'(u)\in L^p(\Omega)$ for all $p<3$ and as a consequence 
the statement in part $(a)$. 
Let $\alpha = \frac{3+\varepsilon}{1+\varepsilon}$ (with $\varepsilon$ as in \eqref{eq1.40:bis}). 
Multiplying \eqref{eq1.10} by $\frac{(f'(u)-f'(0))^{\alpha}}{1+f(u)}$ 
and integrating by parts we obtain
\begin{equation} \label{eq3.10} 
\begin{aligned}
& \int_{\Omega}\frac{f(u)}{1+f(u)} (f'(u)-f'(0))^{\alpha}~dx 
+ 
\int_{\Omega} \frac{f'(u)}{(1+f(u))^2} (f'(u)-f'(0))^{\alpha}a^{ij}(x)u_i u_j ~dx 
\\
&= \alpha \int_{\Omega} \frac{f''(u)}{1+f(u)} (f'(u)-f'(0))^{\alpha-1} a^{ij}(x)u_i u_j ~dx
\\
&\leq 
\alpha  \int_{\{u \leq T\}}\frac{f'(u)^{\alpha-1}f''(u)}{1+f(u)}a^{ij}(x)u_i u_j ~dx 
\\
&\hspace{2cm}+
C\int_{\{u> T\}} \tilde{f}(u)^{(1+\varepsilon)(\alpha-1)-1}f''(u)a^{ij}(x)u_i u_j ~dx
\\
&\leq C\left\{\int_{\{u<T\}}a^{ij}(x)u_i u_j ~dx
+
\int_{\Omega}\tilde{f}(u)f''(u)a^{ij}(x)u_iu_j dx\right\}
\\
&\leq C\left\{T\int_\Omega f(u)~dx
+
\int_{\Omega}\tilde{f}(u)f''(u)a^{ij}(x)u_iu_j dx\right\}. 
\end{aligned}
\end{equation}
In particular, by Lemma \ref{lem1} and the bound \eqref{L1}, we obtain
\[ 
\int_\Omega f'(u)^\alpha~dx\leq C\quad \textrm{where }\alpha =  \frac{3+\varepsilon}{1+\varepsilon}.
\]
Therefore, by the arbitrariness of $\varepsilon>0$, we obtain $f'(u)\in L^p(\Omega)$ for 
all $p<3$. As a consequence, by Corollary \ref{cor2.4} and since $u\in L^1(\Omega)$ 
(see Remark~\ref{remL1}), we obtain the $L^\infty$ estimate established in part $(a)$, 
\textit{i.e.}, if $n<6$ then $\| u\|_{ L^{\infty}(\Omega)}\leq C$.

In the following, we may assume $n\geq 6$. Let us prove now that $f(u)\in L^p(\Omega)$ for 
all $p<n/(n-4)$, and as a consequence, the statement in part $(b)$. 
Now we take $\alpha=1+\frac{1}{1+\varepsilon}$. 
Multiplying \eqref{eq1.10} by $(f'(u)-f'(0))^{\alpha}$  and using \eqref{eq1.40:bis} 
and  Lemma \ref{lem1}, we obtain 
\begin{equation}\label{eq3.60} 
\begin{array}{lll}
\displaystyle \int_{\Omega} f(u)(f'(u)-f'(0))^{\alpha}~dx
&=&
\displaystyle\alpha \int_{\Omega}(f'(u)-f'(0))^{\alpha-1}f''(u) a^{ij}u_i u_j~dx
\\
&\leq& \displaystyle C \int_{\Omega}\tilde{f}(u)f''(u) a^{ij}u_i u_j~dx\leq C.
\end{array}
\end{equation}
Hence, using the convexity of $f$ and that $f'(0)\leq f'(t)/2$ for $t$ large, we obtain
\begin{equation}\label{eq3.30} 
\int_{\Omega} \frac{\tilde{f}(u)^{\alpha+1}}{u^{\alpha}}dx \leq C 
\quad 
\text{for all }\alpha \in (1,2).
\end{equation}

We now repeat the iteration argument of Nedev \cite{Nedev}.
Assume that $u\in L^p(\Omega)$ for all $1\leq p<p_0$. 
Given any positive number $\beta$, set
\[
\Omega_1:=\{x\in \Omega: \frac{\tilde{f}(u)^{\alpha+1}}{u^{\alpha}} >\tilde{f}(u)^{\alpha+1-\beta}\},\,\, 
\Omega_2:=\Omega \setminus \overline{\Omega}_1
=\{x\in \Omega: \tilde{f}(u) \leq u^{\frac{\alpha}{\beta}}\}~.
\]
By \eqref{eq3.30}, we have
\begin{equation} \label{eq3.40}
\int_{\Omega_1}\tilde{f}(u)^{\alpha+1-\beta}~dx \leq C.
\end{equation}
Moreover,
\begin{equation} \label{eq3.50}
\int_{\Omega_2}\tilde{f}(u)^p~dx \leq \int_{\Omega_2}u^{\frac{\alpha}{\beta}p}dx 
\leq C
\quad\text{for all }p<\frac{\beta}{\alpha}p_0.
\end{equation}
Choose $\beta$ such that $\alpha+1-\beta=\frac{\beta}{\alpha} p_0$, 
i.e., $\beta=(\alpha+1)/(1+\frac{p_0}{\alpha})$. 
Then, by \eqref{eq3.40}, \eqref{eq3.50}, and letting $\alpha\uparrow 2$,
we obtain $\tilde{f}(u)\in L^p(\Omega)$ for all $1<p< \frac{3p_0}{2+p_0}$. 
Hence, by elliptic regularity theory (see Proposition~\ref{lem2:1}~(ii)),
\[ 
u\in L^p(\Omega) \quad \text{for all } 1<p<p_1:=\frac{n\frac{3p_0}{2+p_0}}{n-2\frac{3p_0}{2+p_0}}=\frac{3np_0}{2n+(n-6)p_0}.
\]

By \eqref{L1:bis} we can start the iteration process with $p_0=n/(n-2)$.
Set $p_{k+1}:=\frac{3np_k}{2n+(n-6)p_k}$ for $k\geq 1$. Note that 
$p_k\leq n/(n-6)$ for all $k\geq 1$ (by induction and since $p_0=n/(n-2)\leq n/(n-6)$). 
Moreover $p_{k+1}>p_k$ (in fact, this is equivalent to $p_k <n/(n-6)$), and hence, 
$\lim_{k\goto \infty}p_k=n/(n-6)=:p_{\infty}$.

Therefore, we obtain $f(u) \in L^p(\Omega)$ for all $1<p<\frac{3p_{\infty}}{2+p_{\infty}}=\frac{n}{n-4}$. 
The remainder of the statements of Theorem \ref{th1} follow from standard elliptic regularity theory
(see Proposition~\ref{lem2:1}~(ii)).
\end{proof}

\begin{proof}[Proof of Theorem~{\rm \ref{th2}}]
The proof of Theorem \ref{th2} is essentially the same. Using assumption 
\eqref{H:tm2} in \eqref{eq3.10}, the first part of the proof gives $f'(u)\in 
L^\alpha(\Omega)$ with $\alpha={\frac{3-\varepsilon}{1-\varepsilon}}$. Therefore, 
by Corollary~\ref{cor2.4} and since $u\in L^1(\Omega)$ (see Remark~\ref{remL1}), 
$u\in L^{\infty}(\Omega)$ when $\frac{3-\varepsilon}{1-\varepsilon}>\frac{n}{2}$, or equivalently, 
when $n< 6+\frac{4\varepsilon}{1-\varepsilon}$.

Assume $n\geq 6+\frac{4\varepsilon}{1-\varepsilon}$.
To obtain the estimate on $f(u)$, we deduce \eqref{eq3.60} with $\alpha=1+\frac{1}{1-\varepsilon}$ 
using now \eqref{H:tm2} instead of \eqref{eq1.40}. In particular, 
$$
\int_\Omega\frac{\tilde{f}(u)^{\alpha+1}}{u^\alpha}~dx\leq C\quad \textrm{for }\alpha=1+\frac{1}{1-\varepsilon}.
$$
At this point, we repeat the previous iteration argument to obtain the increasing sequence
$$
p_0=\frac{n}{n-2},\quad 
p_{k+1}=\frac{(3-2\varepsilon)np_k}{(2-\varepsilon+(1-\varepsilon)p_k)n-2(3-2\varepsilon)p_k} ,
\quad\textrm{for all }k\geq 0,
$$ 
with limit $p_{\infty}=\frac{(1-\varepsilon)n}{(1-\varepsilon)n-6+4\varepsilon}$. As a consequence, 
$f(u)\in L^p(\Omega)$ for all $p<\frac{\beta}{\alpha}p_\infty$ where $\beta=(\alpha+1)/(1+\frac{p_\infty}{\alpha})$,
i.e., $f(u)\in L^p(\Omega)$ for all $p<\frac{(1-\varepsilon)n}{(1-\varepsilon)n-4+2\varepsilon}$.
The remainder of the statements of Theorem \ref{th2} follow from standard elliptic regularity theory (see 
Proposition~\ref{lem2:1}~(ii)).
\end{proof}

\end{document}